\documentclass[12pt]{amsart}
\usepackage{latexsym, amssymb, amscd, amsfonts, amsthm, color}

 \newlength{\baseunit}               
 \newcount{\numlines}                
\newcommand{\sym}{{\mathrm{sym}}}
\newcommand{\Prob}{{\mathbb{P}}}
\newcommand{\toD}{{\; \stackrel{d}{\longrightarrow} \;}}

\newcommand{\cG}{{\mathcal{G}}}

\newcommand{\SL}{{\mathrm{SL}}}

\newcommand{\zed}{\mathbb{Z}}

\newcommand{\uH}{\mathbb{H}}
\newcommand{\bR}{\mathbb{R}}

\newtheorem{theorem}{Theorem}[section]
\newtheorem{cor}[theorem]{Corollary}
\newtheorem{lemma}[theorem]{Lemma}
\newtheorem{prp}[theorem]{Proposition}
\newtheorem{hypothesis}[theorem]{Hypothesis}
\newtheorem{conjecture}[theorem]{Conjecture}
\theoremstyle{remark}

\newcommand{\notation}[1]{}
\newcommand{\remind}[1]{{}}
\newcommand{\lremind}[1]{{}}

\newcommand{\secretnote}[1]{}

\begin{document}
\pagestyle{plain}
\title{The distribution of the logarithm in an orthogonal and
  a symplectic family of $L$-functions}

\author{Bob Hough}
\email{rdhough@gmail.com}
\maketitle

\begin{abstract}
 We consider the logarithm of the central value $\log L(1/2)$ in the orthogonal
family ${L(s,f)}_{f \in H_k}$ where $H_k$ is the set of weight $k$ Hecke-eigen
cusp form for $SL_2(\zed)$, and in the symplectic family ${L(s,\chi_{8d})}_{d
\asymp D}$ where $\chi_{8d}$ is the real character associated to fundamental
discriminant $8d$. Unconditionally, we prove that the two distributions are
asymptotically bounded above by Gaussian distributions, in the first case of
mean $-1/2 \log \log k$ and variance $\log \log k$, and in the second case of
mean $1/2\log \log D$ and variance $\log \log D$. Assuming both the Riemann and
Zero Density Hypotheses in these families we obtain the full normal law in both
families, confirming a conjecture of Keating and Snaith. 
\end{abstract}

{\parskip=12pt 
\section{Introduction} An important problem in analytic number theory is to
understand the
distribution of values of $L$-functions on the central line $\Re(s) =
\frac{1}{2}$.  Selberg \cite{selberg_zeta} famously proved that as $t$ varies in
large intervals $t \in [T, 2T]$, the real and imaginary parts of the
logarithm of Riemann's zeta function become distributed like independent 
Gaussian random variables.  Since that work, there have been several
efforts to extend the result to a more general setting.  A few years
later, Selberg himself \cite{selberg_dirichlet} proved that for a
fixed value of $t$ the imaginary 
part of $\log L(\frac{1}{2} +it; \chi)$ becomes normally distributed
as $\chi$ varies among Dirichlet characters to a large prime modulus
$q$.  More recently, Bombieri and Hejhal
\cite{bombieri_hejhal} have shown that Selberg's result for zeta is
true for the values $\{L(\frac{1}{2} + it)\}_{t \in [T, 2T]}$ of a
quite general $L$-function, under widely believed assumptions about the zeros of
the function, and Wenzhi Luo \cite{luo} has verified this condition for the
$L$-function associated to any fixed modular form for $SL_2(\zed)$. 

Following the ground-breaking work of Katz and Sarnak \cite{katz_sarnak}, we
now understand the central values $L(\frac{1}{2} + it)$ of an $L$-function as
belonging in a family with a symmetry type governed by one
of the classical compact groups.  The cases considered thus far, of a
fixed $L$-function with argument high in the critical strip, and of
central values of Dirichlet $L$-functions with varying character of
fixed conductor, arise as unitary families. On the basis of
calculations from random matrix theory, Keating and Snaith \cite{keating_snaith}
have
proposed analogous Selberg-type conjectures for the logarithms of central values
of
$L$-functions from families of orthogonal and symplectic symmetry
type, as well.  These conjectures appear far from reach, 
however, because they involve only the real part of the logarithm of
$L$-functions at the fixed point $s = \frac{1}{2}$; even the best
known analytic methods have thus far only succeeded in proving that a positive
proportion of $L$-functions in a family are non-zero at a single
point, and even when the central value is
known to be non-negative, the real part of the logarithm is highly
sensitive to the `low-lying' zeros, near $\frac{1}{2}$, which cannot
presently be controlled.  We will, however, prove partial results  in two
such cases.

Let $S_k,\; k \equiv 0 \bmod{2}$, be the space of weight $k$
modular cusp forms
for $\SL_2(\zed) \backslash \uH$ and let $H_k$ be its basis of
$\sim \frac{k}{12}$ simultaneous eigenvectors of the Hecke
operators, normalized to have first Fourier coefficient equal to 1.
Let $f \in H_k$ have Fourier expansion 
\[f(z) = \sum_{n = 1}^\infty n^{\frac{k-1}{2}}\lambda_f(n)e(nz). \] 
The  $L$-function $L(s,f)$ associated to $f$ is then
\begin{equation}\label{euler_product} L(s,f) = \sum_{n = 1}^\infty
\frac{\lambda_f(n)}{n^s} = \prod_{p}\left(1 -
\frac{\lambda_f(p)}{p^s} + \frac{1}{p^{2s}}\right)^{-1}.\end{equation}
This has completed $L$-function \[\Lambda(s,f) =
(2\pi)^{-s}
\Gamma\left(s + \frac{k-1}{2}\right) L(s,f),\] which satisfies the self-dual
functional equation \[\Lambda(s,f) = i^k\Lambda(1-s,f).\]  When $k
\equiv 2\bmod 4$ this means that the central value $L(\frac{1}{2}, f)
= 0$, so for $k \equiv 0 \bmod 4$ we consider the family of values
$\{L(\frac{1}{2},f)\}_{f \in H_k},$ which is expected to have
orthogonal symmetry type.  These central values have a certain extra
significance because Kohnen and Zagier proved the striking formula
\[L\left(\frac{1}{2},f\right) = \frac{c_g(1)^2\pi^k}{(k-1)!} \frac{\langle f,
  f\rangle}{\langle g, g\rangle} \] relating the central value
$L(\frac{1}{2},f)$ to the Petersson norm of $f$ and the first Fourier
coefficient and Petersson norm of a
half-integral weight form $g$ that lifts to $f$ under the Shimura
correspondance. A particular consequence is that $L(\frac{1}{2}, f) \geq0$; this
is
one of
the few families of $L$-functions where non-negativity of the central value is
known, although see \cite{lapid} for a number of further examples.

As a second example we let $d > 0$  be a
fundamental discriminant with associated quadratic character
$\chi_{d}(n) = \left(\frac{d}{n}\right)$ of conductor $d$.  The
corresponding  Dirichlet $L$-function is \[L(s, \chi_{d}) =
\sum_{n =
  1}^\infty \frac{\chi_{d}(n)}{n^s} = \prod_p \left(1 -
  \frac{\chi_{d}(p)}{p^s}\right)^{-1},\qquad \qquad \Re(s) > 1\] with
completed $L$-function \[\Lambda(s, \chi_{d}) =
\left(\frac{d}{\pi}\right)^{\frac{s+a}{2}} \Gamma\left(\frac{s+a}{2}\right) L(s,
\chi_{d}), \qquad a = \frac{1 - \chi_d(-1)}{2}.\]  This also satisfies the
self-dual functional equation
\[\Lambda(s, \chi_d) =  \Lambda(1-s, \chi_{d})\] and conjecturally
$L(\frac{1}{2}, \chi_{d}) \geq 0$, but this  
is not known.  For convenience we consider the family of central values
$\{L(\frac{1}{2}, \chi_{8d})\}_{d \in s(D)}$ where $s(D)$ denotes the
set of squarefree and odd  $d$, $\frac{D}{2} < d \leq D$. In particular
$\chi_{8d}(-1) = 1$ so that $a = 0$ above. This is
expected to be a family exhibiting symplectic symmetry.

We have two primary results.  The first result proves,
unconditionally, `one-half' of the Keating-Snaith conjectures.
\begin{cor}\label{upper_bound}
Let $k \equiv 0 \bmod{4}$.  As $k \to \infty$ we have
\begin{align*}\Prob\left[f \in H_k: \frac{1}{\sqrt{\log \log k}}\left(\log
    L\left(\frac{1}{2},f\right)+\frac{1}{2}\log \log k\right) > A\right]\qquad &
\\ \leq
  \frac{1}{\sqrt{2\pi}}\int_A^\infty e^{-\frac{x^2}{2}} dx +
o_A(1).&\end{align*}
In particular, for any fixed $\epsilon >0$, $L(\frac{1}{2}, f) < (\log
k)^{-\frac{1}{2} + \epsilon}$ with probability $1- o_\epsilon(1)$.
Also, as $D \to \infty$,
\begin{align*}
 \Prob\left[d \in s(D): \frac{1}{\sqrt{\log \log D}}\left(\log
    \left|L\left(\frac{1}{2},\chi_{8d}\right)\right|-\frac{1}{2}\log \log
D\right) > A\right] \qquad&\\ \leq
  \frac{1}{\sqrt{2\pi}}\int_A^\infty e^{-\frac{x^2}{2}} dx + o_A(1).&
\end{align*}
\end{cor}
\noindent In \cite{sound_moments}, Soundararajan made the basic
observation that, on the Riemann Hypothesis, while zeros near $\frac{1}{2}+it$
can greatly alter the
value of $\log |\zeta(\frac{1}{2}+it)|$, they tend to decrease its
value as compared with that of $\log |\zeta(\frac{1}{2} + \sigma +
it)|$ at points off the critical line.  Our proof of Corollary
\ref{upper_bound} is based upon an unconditional version of this fact,
together with the following slightly  technical result. 
\begin{theorem}\label{right_of_half}
 Let $\sigma = \sigma(k)$ be a function of $k$, tending to 0 as $k \to
\infty$ in such a way that $\sigma \log k \to \infty$ but $\frac{\sigma \log
k}{\sqrt{\log \log k}}\to 0$.  Also, for $f \in H_k$ put \[A(f) =
\frac{1}{\sqrt{\log \log k}}\left(\log \left|L\left(\frac{1}{2} + \sigma,
f\right)\right| +
  \frac{1}{2} \log \log k\right).\]   Then 
\[\frac{1}{|H_k|} \sum_{f \in H_k} \delta_{A(f)} \to N(0, 1), \qquad k
\to \infty.\]  Here
$\delta_x$ is the point mass at $x$, 
$N(0,1)$ is the standard normal distribution, and the convergence is
in the sense of distributions.

Similarly, let $\sigma = \sigma(D)$ be a function of $D$, tending to 0
as $D \to \infty$ in such a way that $\sigma \log D \to \infty$ but
$\frac{\sigma \log D}{\sqrt{\log \log D}} \to 0$.  For $d \in s(D)$,
put \[A(d) = \frac{1}{\sqrt{\log \log D}} \left(\log \left|L\left(\frac{1}{2} +
  \sigma, \chi_{8d}\right)\right| - \frac{1}{2}\log \log D\right).\]  Then
\[\frac{1}{|s(D)|} \sum_{d \in s(D)} \delta_{A(d)} \to N(0,1), \qquad
D \to \infty.\]
\end{theorem}
\noindent This Theorem is proven using Selberg's method in
\cite{selberg_dirichlet}; in particular it makes use of 
`zero-density' estimates putting almost all of the low-lying zeros of
the corresponding $L$-functions very near the half-line.  In the case
of $L(s, \chi_{8d})$, such a result is essentially available from the
work of Conrey and Soundararajan in \cite{conrey_sound}.  For the case of
$L(s,f)$  see
\cite{hough_zero_density}.

For our second main result we assume some weak conjectural information
about the low-lying zeros in the families $\{L(s,f)\}_{f \in H_k}$,
and $\{L(s, \chi_{8d})\}_{d \in s(D)}$ in order to deduce the full
Keating-Snaith conjectures for these families. Given $f \in H_k$ and
$s$ near $\frac{1}{2}$,  $L(s,f)$ has conductor $\asymp k^2$, and
therefore for $1 \ll T = k^{o(1)}$ the number of zeros of $L(s,f)$ up
to height $T$ grows as $\frac{T}{\pi}\log k$.  Reasoning
probabilistically, we
might then expect that for most $f \in H_k$,
$\gamma_{min}(f) \gg \frac{1}{\log k}$, where  \[\gamma_{min}(f) =
\min_{\rho =\frac{1}{2} +  \beta + i\gamma} |\gamma|\] is the height of the
lowest
non-trivial zero of $L(s,f)$.  Similarly, for $d \in s(D)$ and $s$
near $\frac{1}{2}$, $L(s,
\chi_{8d})$ has conductor $\asymp D$, and therefore we might typically
expect that $\gamma_{min}(d) \gg \frac{1}{\log D}$.  We formalize this
heuristic in the following hypothesis.
\begin{hypothesis}[Low-lying Zero Hypothesis]\label{low_zero_hypothesis}
 Assume $y = y(k) \to \infty$ with $k$.  Then
\[\Prob\left[ f \in H_k: \gamma_{min}(f) < \frac{\pi}{y \log k}\right]=o(1),
\qquad k \to \infty.\]  Similarly, if $y = y(D) \to \infty$ with $D$
then
\[\Prob\left[ d \in s(D): \gamma_{min}(d) < \frac{2\pi}{y \log
    D}\right] = o(1), \qquad D \to \infty.\]
\end{hypothesis}
\noindent In fact, stronger and more detailed statements about the low-lying
zeros in these two families are expected to be true.  Specifically,
Iwaniec, Luo and Sarnak \cite{iwaniec_luo_sarnak} have conjectured
that in essentially any natural family of $L$-functions of conductor $C$, the
one-level density of
zeros at a scale of $\frac{2\pi}{\log C}$ depends asymptotically only
on the symmetry type of the family.  For our two families of
$L$-functions, their `Zero Density Conjecture' takes the following shape.
\begin{conjecture}[Zero Density Conjecture]
 Let $\phi(x)$ be a Schwarz class function on $\bR$ with Fourier transform
having compact support.  Define the densities \[W(SO_{\text{even}})(x)dx =
\left(1 +
\frac{\sin 2\pi x}{2\pi x}\right)dx, \qquad W(Sp)(x)dx = \left(1 - \frac{\sin
  2\pi x}{2\pi x}\right) dx\] and write the non-trivial zeros of 
$L(s)$ as $\rho = \frac{1}{2} + i\gamma$, with $\gamma$ possibly complex if
the Riemann Hypothesis for $L(s)$ is false.  Then
\[\lim_{\substack{k \to \infty \\ k\equiv 0 \bmod 4}} \frac{1}{|H_k|}{\sum_{f
\in H_k}}
\sum_{\substack{\Lambda(\rho,f)=0\\ \rho = \frac{1}{2} + i\gamma}}
\phi\left(\frac{\gamma \log k}{\pi}\right) = \int_{-\infty}^\infty \phi(x)
W(SO_{\text{even}})(x)dx\]
and 
\[\lim_{D \to \infty} \frac{1}{|s(D)|} \sum_{d \in s(D)}
\sum_{\substack{\Lambda(\rho, \chi_{8d}) = 0\\ \rho = \frac{1}{2} +
    i\gamma}}\phi\left(\frac{\gamma \log D}{2\pi}\right) =
\int_{-\infty}^\infty \phi(x)W(Sp)(x)dx.\]
\end{conjecture}
\noindent It is a straightforward exercise to prove that our Low-lying
Zero Hypothesis is implied by the Zero Density Conjecture together
with the Riemann Hypothesis for the corresponding family of
$L$-functions.

We now state our second main result.
\begin{theorem}\label{conditional}
Suppose the Low-lying Zero Hypothesis holds for
$\{L(s,f)\}_{f \in H_k}$.  For $f \in H_k$ put \[B(f) = 
\frac{1}{\sqrt{\log \log k}}\left(\log L\left(\frac{1}{2},f\right) +
  \frac{1}{2}\log \log k\right).\]  Then, as distributions
\[ \frac{1}{|H_k|} \sum_{f \in H_k} \delta_{B(f)} \to N(0,1), \qquad k
\to \infty.\]   

Similarly, assume the Low-lying Zero Hypothesis for $\{L(s,
\chi_{8d})\}_{d \in s(D)}$ and for $d \in s(D)$ put \[B(d) =
\frac{1}{\sqrt{\log \log D}}\left(\log \left|L\left(\frac{1}{2},
\chi_{8d}\right)\right| -
    \frac{1}{2}\log \log D\right).\]  Then, in the sense of distributions,
\[ \frac{1}{|s(D)|} \sum_{d \in s(D)} \delta_{B(d)} \to N(0,1), \qquad
D \to \infty.\] 

In particular, either of these results is true if both the Riemann
Hypothesis and the Zero Density Hypothesis is true for the
corresponding family of $L$-functions.
\end{theorem}

\section{Background}\label{lemma_section}
 In this section we collect together standard facts regarding our two families
of $L$-functions, as well as the part of Selberg's work that we need for our
arguments.

\subsection{$L$-function coefficients, and orthogonality}
 For $f\in H_k$, the Fourier coefficients of $f$ satisfy the
Hecke relations
\begin{equation}\label{hecke_relations} \lambda_f(m)\lambda_f(n) =
\sum_{d|(m,n)} \lambda_f\left(\frac{mn}{d^2}\right).\end{equation}
A specific consequence of this fact is that for distinct primes $p_1,
..., p_r$ we have
\begin{equation}\label{prime_product}\lambda_f(p_1)^{e_1}\lambda_f(p_2)^{e_2}
\cdots
\lambda_f(p_r)^{e_r} =
\sum_{0 \leq j_1 \leq \lfloor\frac{e_1}{2}\rfloor} \cdots \sum_{0 \leq
  j_r \leq \lfloor\frac{e_r}{2}\rfloor}c(\mathbf{e},
\mathbf{j})\lambda_f(p_1^{e_1 - 2j_1}
\cdots p_r^{e_r - 2 j_r})\end{equation} for some positive coefficients
$c(\mathbf{e}, \mathbf{j})$.
\begin{lemma}
 We have $c(\mathbf{2}, \mathbf{*}) = 1$ where $\mathbf{2}$ is the string
consisting entirely of 2's and  $\mathbf{*}$ is any string containing only 0's
and
1's. Also, for general $\mathbf{e}, \mathbf{j}$,  $c(\mathbf{e}, \mathbf{j})
\leq 2^{e_1 + ... + e_r}.$
\end{lemma}
\noindent Recall, also, Deligne's bound $|\lambda_f(n)| \leq d(n)$.

We use the following basic orthogonality relation on $H_k$.
\begin{lemma}\label{H_k_orthogonality}
Let $0 < \delta < 2$.  There exists $\gamma = \gamma(\delta)>0$ such
that if $m < k^{2-\delta}$ then
\[\frac{1}{|H_k|}\sum_{f\in H_k} \lambda_f(m) = \frac{\delta_{m =
    \square}}{\sqrt{m}} + O_\delta(k^{-\gamma}).\]
\end{lemma}
\begin{proof}
Actually, this is a combination of two different estimates.  Using the
Petersson Trace Formula, Rudnick and Soundararajan
(\cite{rudnick_sound_examples}, Lemma 2.1)  prove that for $mn <
\frac{k^2}{10000}$, 
\[\sum_{f \in H_k} \frac{2\pi^2}{k-1}L(1, \sym^2 f)^{-1}
\lambda_f(m)\lambda_f(n) =
\delta_{m=n} + O(e^{-k}).\]  Here $w_f = \frac{2\pi^2}{k-1}L(1,\sym^2 f)^{-1}$
is
the  `harmonic weight' of $f$, and $L(s, \sym^2 f)$ is the symmetric
square $L$-function attached to $L(s,f)$ with coefficients given by
\[L(s,\sym^2 f) = \sum_{n = 1}^\infty \frac{\rho_f(n)}{n^s} =
\zeta(2s)\sum_{n =1}^\infty \frac{\lambda_f(n^2)}{n^s}, \qquad \Re(s)
> 1.\]

A now-standard method of Kowalski-Michel (\cite{kowalski_michel},
Proposition 2) allows the removal of the harmonic weight by truncating
the Dirichlet series for $L(1, \sym^2 f)$; with $x = \frac{k^{\delta/2}}{100}$
and recalling $|H_k| = \frac{k-1}{12} + O(1)$, their method
gives
\begin{align*}\frac{1}{|H_k|} \sum_{f \in H_k} \lambda_f(m) &= \frac{1}{|H_k|}
\sum_{f \in H_k} w_f \frac{k-1}{2\pi^2} L(1, \sym^2 f)
\lambda_f(m) \\&= \frac{1}{\zeta(2)} \sum_{f \in H_k} w_f \lambda_f(m)
\sum_{\ell^2 d < x} \frac{\lambda_f(d^2)}{\ell^2 d} +
O(k^{-\gamma}).\end{align*}
Substituting the bound of Rudnick and Soundararajan, one deduces the lemma.
\end{proof}

For the real characters $\chi_{8d}$, our basic orthogonality relation is the
following.
\begin{lemma}
 Let $n < D^{2-\delta}$.  Then there is $\gamma = \gamma(\delta)>0$ such that 
\[ \frac{1}{|s(D)|}\sum_{d \in s(D)} \left(\frac{8d}{n}\right) = \delta_{n =
\square}
\prod_{\substack{p|n \\ \text{odd}}} \left(\frac{p}{p+1}\right) +
O(D^{-\gamma}).\]
\end{lemma}
\begin{proof}
 Note that $\mu(2d)^2$ is  the indicator function for odd,
squarefree $d$.  Rudnick and Soundararajan (\cite{rudnick_sound_examples} Lemma
3.1) prove, for any $z > 3$, that if $n$ is a perfect square then
\[\sum_{d \leq z} \mu(2d)^2 \left(\frac{8d}{n}\right) = \frac{z}{\zeta(2)}
\prod_{p|2n}\left(\frac{p}{p+1}\right) + O(z^{\frac{1}{2} + \epsilon}
n^\epsilon)\]
and if $n$ is not a square then
\[\sum_{d \leq z} \mu(2d)^2 \left(\frac{8d}{n}\right) = O(z^{\frac{1}{2}
}n^{\frac{1}{4}})
\log(2n).\]
The result follows on taking successively $z = D/2, D$.
\end{proof}

\subsection{Selberg's work: two expressions for the logarithm}
Writing the Euler product of $L(s,f)$ as
\[
L(s,f) = \prod_p \left(1 - \frac{\lambda_f(p)}{p^s} +
\frac{1}{p^{2s}}\right)^{-1} = \prod_{p} \left(1 -
\frac{\alpha_p}{p^s}\right)^{-1}\left(1-
\frac{\overline{\alpha_p}}{p^s}\right)^{-1}, \qquad  \Re(s) > 1
\]
we have that for $m = 1, 2, ...$
\[ \lambda_f(p^m) = \alpha_p^m + \alpha_p^{m-2} + ... + \alpha_p^{-m+2} +
\alpha_p^{-m} \] where for each $p$, $\alpha_p$
is a complex number of modulus 1 solving
$\alpha_p + \overline{\alpha_p} = \lambda_f(p)$. Logarithmically differentiating
$L(s,f)$ term-by-term we obtain
\begin{equation}\label{def_Lambda_f} -\frac{L'}{L}(s,f) = \sum_{n = 1}^\infty
\frac{\Lambda_f(n)}{n^s} = \sum_{m = 1}^\infty \sum_p \frac{(\alpha_p^m +
\overline{\alpha}_p^m)\log p}{p^{ms}}, \qquad \qquad \Re(s)
> 1.\end{equation} In particular, $\Lambda_f(n)$ is supported on prime powers,
and is given explicitly by 
\begin{equation}\label{Lambda_f_value} \Lambda_f(p^m) = (\lambda_f(p^m) -
\lambda_f(p^{m-2})) \log p, \qquad \qquad m \geq 1,\end{equation} with the
convention that $\lambda_f(p^{-1}) = 0$.

Similarly we have
\[L(s, \chi_{8d}) = \prod_p \left(1 - \frac{\chi_{8d}(p)}{p^s}\right)^{-1},
\qquad \Re(s) > 1\] and logarithmically differentiating this leads to
\begin{equation}\label{def_Lambda_8d} -\frac{L'}{L}(s, \chi_{8d}) = \sum_n
\frac{\Lambda_{8d}(n)}{n^s}, \qquad \Re(s) > 1 \end{equation} with
$\Lambda_{8d}$ supported on primes powers
and \begin{equation}\label{Lambda_8d_value}\Lambda_{8d}(p^n) =
\left(\frac{8d}{p^n}\right) \log p. \end{equation}

In a standard way one may write down an expression for $-\frac{L'}{L}(s,*)$
similar to (\ref{def_Lambda_f}) and (\ref{def_Lambda_8d}) when $\frac{1}{2} <
\Re(s) \leq 1$, although in this case the zeros of $L(s,*)$ enter into
the formula. The following lemma is the analog of \cite{selberg_zeta}, Lemma
10 with $L(s,*)$ replacing the Riemann zeta function.
\begin{lemma} Let $*$ stand in for either $8d$ or $f$, so that $L(s,*)$ is
either $L(s, \chi_{8d})$ for some $d \in s(D)$ or $L(s, f)$ for some $f \in
H_k$.  

Let $x > 1$ be a parameter and define
\[ \Lambda_{x,*}(n) = \Lambda_*(n) a_x(n); \qquad \qquad a_x(n)  =
\left\{\begin{array}{lll} 1,&& 1 \leq n
\leq x\\  \frac{\log^2\frac{x^3}{n} - 2 \log^2 \frac{x^2}{n}}{2
\log^2 x}, && x \leq n \leq x^2\\\frac{\log^2
\frac{x^3}{n}}{2 \log^2 x}, && x^2 \leq n \leq x^3 \end{array}\right..
\]
For $s$ not coinciding with a trivial or non-trivial zero of $L(s,*)$ we have
\begin{align}\label{euler_derivative}-\frac{L'}{L}\left(\frac{1}{2} + s,*\right)
= \sum_{n \leq x^3}
\frac{\Lambda_{x,*}(n)}{n^{\frac{1}{2} + s}} &+ \frac{1}{\log^2
x}
\sum_{\substack{\rho:\; \Lambda(\rho,*) = 0\\ \text{non-trivial}}}\frac{x^{\rho
- \frac{1}{2} - s}(1 - x^{\rho - \frac{1}{2} -
s})^2}  {(\frac{1}{2} +s - \rho)^3} \\ \notag &+ \frac{1}{\log^2 x}
\sum_{\substack{q:\; L(-q,*) = 0,\; \Lambda(-q, *) \neq 0\\ \text{trivial}}}
 \frac{x^{-q-\frac{1}{2}-s}(1 - x^{ - q-\frac{1}{2} -
s})^2}{(\frac{1}{2}+ q + s)^3}.\end{align}
\end{lemma}

\begin{proof}
 The sum $\sum_{n \leq x^3} \frac{\Lambda_{x,*}(n) }{n^{\frac{1}{2}
+
s}}$ is the result of expanding $-\frac{L'}{L}(z,*)$ in its
Dirichlet series
in
\[\frac{1}{2\pi i \log^2 x}\int_{(3)} \frac{x^{z - \frac{1}{2} -s}(1 - x^{z
- \frac{1}{2} - s})^2}{(z- \frac{1}{2} -
s)^3}\left(-\frac{L'}{L}(z,*)\right)dz\] 
and integrating term-by-term.
The remainder of the expression is obtained by shifting the $z$-contour
leftward and evaluating residues.
\end{proof}

Introduce the gamma factors \begin{equation}\label{def_archimedian_factor}
\gamma_f(s) =
(2\pi)^{-s}\Gamma\left(s + \frac{k-1}{2}\right), \qquad \qquad \gamma_{8d}(s) =
\left(\frac{8d}{\pi}\right)^{-\frac{s}{2}}\Gamma\left(\frac{s}{2}\right)
\end{equation}
so that we may write in a unified way
\[\Lambda(s, *) = \gamma_*(s) L(s, *)\] for the completed $L$-function
corresponding to either $L(s,f)$ or $L(s, \chi_{8d})$.  The gamma factors do
not play a significant role in our results; we only need
\begin{equation}\label{log_deriv_gamma} \frac{\gamma_f'}{\gamma_f}(\sigma) =
\log k + O(1), \qquad \qquad \frac{\gamma_{8d}'}{\gamma_{8d}}(\sigma) =
\frac{1}{2}
\log D + O(1), \end{equation}
uniformly in $\frac{1}{2} \leq \sigma \leq 1$.
The completed
$L$-function is entire of order 1 and hence has a Hadamard product running
over its zeros, 
\begin{equation}\label{hadamard_product}\Lambda(s,*) = e^{A +
Bs}\prod_{\rho: \; \Lambda(\rho, *) = 0}\left(1 -
\frac{s}{\rho}\right)e^{\frac{s}{\rho}}.\end{equation}
Logarithmically differentiating $\Lambda(s,*)$, and using that \[B =
-\sum_{\rho: \Lambda(\rho, *) = 0}\Re\frac{1}{\rho}\] as in \cite{davenport}, p.
82, one proves the following
lemma.

\begin{lemma}\label{hadamard}
 For real $\sigma > 0$ we have
\begin{equation}\label{hadamard_derivative}
 -\frac{L'}{L}\left(\frac{1}{2} + \sigma,*\right) = 
\frac{\gamma_*'}{\gamma_*}\left(\frac{1}{2} + \sigma\right)  -
\sum_{\substack{\rho=
\frac{1}{2} + \beta + i \gamma\\ \Lambda(\rho, *) = 0}} \frac{ \sigma - \beta}{(
\sigma - \beta)^2 + \gamma^2}.
\end{equation}
\end{lemma}

One of Selberg's major achievements in \cite{selberg_zeta} was that he gave an
efficient way to compute
$\log\zeta(\frac{1}{2} + it)$ as a short sum over primes. By balancing the
expression for $-\frac{\zeta'}{\zeta}(s)$ coming from
the Hadamard product as in (\ref{hadamard_derivative}) against the expression
from the Euler product (\ref{euler_derivative}), he was able to bound the
contribution of the zeros in (\ref{euler_derivative}).  To do so, Selberg
introduced a perturbation $\sigma_{x,t}$ depending on the
location of the zeros of $\zeta$ near height $t$, and evaluated $\log
\zeta(\frac{1}{2} + \sigma_{x,t} + it)$ in place of $\log \zeta(\frac{1}{2} +
it)$.

For $\log |L(\frac{1}{2},*)|$ the analog of 
$\sigma_{x,t}$
is
\begin{align}\label{def_sigma}&\sigma_{x,*} = 2\max_{\rho\in
     \cG_{x,*}}\left(\beta, \frac{2}{\log x}\right);\\ \notag &\cG_{x,*} =
  \left\{\rho = \frac{1}{2} + \beta + i\gamma : \Lambda(\rho, *) = 0, |\gamma|
\leq \frac{x^{3|\beta|}}{\log x},\; |\beta|
\geq  \frac{2}{\log x}\right\}.\end{align}
Selberg's argument for $\log \zeta(\frac{1}{2} + \sigma_{x,t} + it)$ carries
over with trivial modifications to bound the zero sum of $L(s,*)$ at
$s = \frac{1}{2} + \sigma_{x,*}$ and thus to the evaluation of
$\log L(\frac{1}{2} +
\sigma_{x,*},*);$ the result is the
following lemma. 

\begin{lemma}\label{log_L_off}
Let $C = k^2$ for $L(s,f)$ or $C = 8d$ for $L(s, \chi_{8d})$ be the conductor
of the $L$-function near $s = \frac{1}{2}$.  We have
\begin{equation}\label{zero_sum_bound}
 \sum_{\substack{\rho = \frac{1}{2} + \beta + i\gamma\\ \Lambda(\rho,*) = 0}}
\frac{\sigma_{x,*}}{(\sigma_{x,*} - \beta)^2 + \gamma^2} = O\left(\left|\sum_{n
< x^3} \frac{\Lambda_{x,*}(n)}{n^{\frac{1}{2} + \sigma_{x,*}}} \right|\right) +
O(\log C)
\end{equation}
and
\begin{equation}\label{log_prime_sum}\log L(\frac{1}{2} + \sigma_{x,*},*) =
\sum_{n \leq x^3} \frac{\Lambda_{x,*}(n)}{n^{\frac{1}{2} +
\sigma_{x,*}} \log n} + O\left(\frac{1}{\log x} \left|\sum_{n \leq x^3}
\frac{\Lambda_{x,*}(n) }{n^{\frac{1}{2} +
\sigma_{x,*}}}\right|\right)
+ O\left(\frac{\log C}{\log x}\right).\end{equation}
\end{lemma}
\begin{proof}
 See  \cite{selberg_zeta} pp 22-26.
\end{proof}

In order to proceed further with Selberg's approach we need an understanding
of the perturbation $\sigma_{x,*}$, that is, we need input regarding the
distribution of zeros of $L(s,*)$ near the central point $s =
\frac{1}{2}$ as either $f$ varies in $H_k$ or $d$ varies in $s(D)$. Our basic
analytic ingredient is
the following.

\begin{theorem}\label{zero_density}
 For a sufficiently small $\delta > 0$ there exists $\theta = \theta(\delta)$
such that, uniformly in  $ \frac{2}{\log k}< \sigma < \frac{1}{2}$ and 
$\frac{10}{\log k}< T <
k^{2\delta},$
\begin{align*} N(\sigma, T, k) \stackrel{\text{def}}{=}& \frac{1}{|H_k|}{\sum_{f
\in H_k}}
\#\left\{L(\frac{1}{2} + \beta + i \gamma, f) = 0: \sigma < \beta, \; |\gamma|<
T\right\} \\=& O(T k^{ -2 \theta \sigma } \log k),\end{align*} and also,
uniformly in
$\frac{4}{\log D} < \sigma < \frac{1}{2}$ and $\frac{10}{\log D} < T <
D^\delta$,
\begin{align*} N(\sigma, T, D) \stackrel{\text{def}}{=}&
\frac{1}{|s(D)|}{\sum_{d \in
s(D)}}
\#\left\{L(\frac{1}{2} + \beta + i \gamma, \chi_{8d}) = 0: \sigma < \beta, \;
|\gamma|<
T\right\}\\ =& O(T D^{ - \theta \sigma } \log D).\end{align*}
\end{theorem}

\begin{proof}
This result for $\{L(s,f)\}_{f \in H_k}$ is proved in
\cite{hough_zero_density}. 

The details of the second statement are
largely contained in \cite{conrey_sound}, but the situation is slightly
different, so we briefly sketch the argument.  The essential ingredient  is an
asymptotic evaluation of the twisted second moment
\[ \frac{1}{|s(D)|} \sum_{d \in s(D)} \left|L(\frac{1}{2} + \sigma + it; \chi_{8
d})\right|^2\chi_{8d}(\ell) = (\text{asymptotic main term}) +
O(D^{-\kappa}t^A\ell^B)\] uniform in the range $\sigma > 0$ and $t < D^\delta$,
 where $\delta, \kappa$ and $A, B$ are fixed positive constants.  The
twisted second moment with power-saving error term was first obtained in this
family at the central point
$\sigma = t=0$ in \cite{sound_non_vanishing}, and in Propostion 
2.3 of \cite{conrey_sound} the asymptotic is given
for  the range $0 < \sigma  = O(1)$, $|t| = O(1)$   for the closely related
family $\{L(s,
\chi_{-8d})\}_{d \in
s(D)}$.  The extension of that result to the range $t = O(D^\delta)$ incurs no
further difficulties; the limiting factor is the size of the analytic
conductor $(|t| + D)$ of the family of $L$-functions, which in this case is
essentially 
unchanged for $t$ as large as $D^{1-\epsilon}$.  The authors in
\cite{conrey_sound} also remark that
their result remains valid in any arithmetic progression of fundamental
discriminants,  in particular, for the family $\{\chi_{8d}\}_{d \in s(D)}$
considered here.  

From an asymptotic formula for the twisted second moment 
it is a standard, albeit somewhat laborious, task to bound the
mean-square \[\frac{1}{|s(D)|} \sum_{d \in s(D)} \left|\eta(\frac{1}{2} +
\sigma + it, \chi_{8d})\right| ^2 \leq 1 + O(D^{-\theta \sigma}) +
O(D^{-\frac{\kappa}{2}} |t|^A), \qquad |t| < D^\delta\] where $\eta(s,
\chi_{8d}) = L(s, \chi_{8d}) M(s,
\chi_{8d})$  with $M(s, \chi_{8d})$  a 
 short mollifying Dirichlet polynomial.  The proof is then completed by
appealing to a version of the argument principle to bound  the total
number
of zeros of $\{\eta(s, \chi_{8d})\}_{d \in s(D)}$ in the specified box.  This
entire program is carried out for
the family $\{L(s, \chi_{-8d})\}_{d \in s(D)}$ in the most difficult range
where the box has height $T = O(\frac{1}{\log D})$ in \cite{conrey_sound}.  The
method was first introduced in \cite{selberg_dirichlet} and full
details are contained there for the family of all Dirichlet $L$-functions to a
fixed prime conductor. For another example calculation,
see \cite{hough_zero_density}. 
\end{proof}

As a consequence we derive the following essential lemma.
\begin{lemma}\label{zero_prob}
Let $\mathcal{F}$ be either the family of $L$-functions 
$\mathcal{F} = \{L(s,f)\}_{f \in H_k}$ or the family $\mathcal{F} = \{L(s,
\chi_{8d})\}_{d
\in s(D)}$. Denote $\Prob$ the uniform probability on $\mathcal{F}$.  Let $C =
k^2$ or $C = D$ be the respective conductor of the family. 

For $x = x(C)$ growing with $C$ in such a way that $ \frac{\log x}{\log C}
\to 0$ as $C \to \infty$ we have
\[ \Prob\left[ \exists \; \rho = \frac{1}{2} + \beta + i\gamma:
\Lambda(\rho,*) = 0,\; \beta > \frac{4}{\log x}, \; |\gamma| \leq
\frac{x^{3\beta}}{\log x}\right] = o(1)\] as $C \to \infty$.
\end{lemma}

\begin{proof}
 We may evidently assume that $\log x$ is larger than a fixed constant, and
less than a sufficiently small constant times $\log C$.  Then
\begin{align*}
 \Prob\left[ \exists \; \rho :
\; \beta > \frac{4}{\log x}, \; |\gamma| \leq
\frac{x^{3\beta}}{\log x}\right] & \leq \Prob\left[\bigcup_{j =
4}^{\lceil \frac{\log x}{2}\rceil} \left\{ \exists \; \rho :
\; \beta > \frac{j}{\log x}, \; |\gamma| \leq
\frac{e^{3(j+1)}}{\log x}\right\}\right]
\\ & \leq \sum_{j = 4}^{\lceil \frac{\log x}{2}\rceil}
\Prob\left[\exists \; \rho :
\; \beta > \frac{j}{\log x}, \; |\gamma| \leq
\frac{e^{3(j+1)}}{\log x}\right].
\end{align*}
By applying Theorem \ref{zero_density}, the last sum is bounded by 
\begin{align*} \ll \sum_{j = 4}^{\lceil \frac{\log x}{2}\rceil}
\frac{e^{3(j+1)}}{\log x} C^{-\frac{\theta j}{\log x}} \log C& \leq \frac{\log
C}{\log x} \sum_{j = 4}^{\lceil \frac{\log x}{2}\rceil}
e^{\left(-\theta \frac{\log
C}{\log x} + 3\right) j}  \ll \frac{\log C}{\log x} e^{\frac{-\theta \log C}{
\log x}},
\end{align*} 
and this tends to zero as $C \to \infty$.
\end{proof}

\subsection{Convergence in the sense of distributions}
Before turning to the main argument, we record, for repeated later use, the
following simple fact concerning convergence in the sense of distributions.

Suppose we have a sequence
of finite sets $\{R_n\}$. For each $n$ let there be two functions $f,
\tilde{f}:R_n \to \bR$, so that we obtain two
sequences of probability measures $\{\mu_n\}$, $\{\tilde{\mu}_n\}$ on $\bR$,
\[\mu_n = \frac{1}{|R_n|}\sum_{s \in R_n} \delta_{f(s)}, \qquad \qquad
\tilde{\mu}_n =
\frac{1}{|R_n|}\sum_{s \in R_n} \delta_{\tilde{f}(s)}.\]  
\begin{lemma}[Distribution comparison lemma]\label{conv_dist_lemma}
Let $\mu$ be a finite (Borel)
measure on $\bR$.  Each of the following three conditions is sufficient to
guarantee the simultaneous convergence in distribution
\[ \mu_n \toD \mu \qquad \Leftrightarrow \qquad \tilde{\mu}_n
\toD \mu\]
 of $\mu_n$ and $\tilde{\mu}_n$ to $\mu$.
\begin{align*}
\tag{i}  \frac{1}{|R_n|}\sum_{\substack{s \in R_n \\ f(s) \neq \tilde{f}(s)}} 1
=
o(1), \qquad\qquad &n \to \infty \\
\tag{ii}  \sup_{s \in R_n} |f(s) - \tilde{f}(s)| = o(1), \qquad\qquad &n
\to \infty\\
\tag{iii}  \frac{1}{|R_n|}\sum_{s \in R_n}  |f(s) - \tilde{f}(s)|^2 =
o(1),\qquad\qquad & n
\to \infty.
\end{align*}
\end{lemma}

\section{The distribution of the prime sum}
We first show that the short prime sums $(x = C^{o(1)})$
\[\left\{\sum_{n \leq x}
\frac{\Lambda_f(n)}{\sqrt{n}}\right\}_{f \in H_k}, \qquad  \left\{\sum_{n \leq
x}
\frac{\Lambda_{8d}(n)}{\sqrt{n}}\right\}_{d \in S(D)}\] converge to the
appropriate
Gaussian distributions as the conductor $C \to \infty$.  The
main work will
then be in comparing $\log |L(\frac{1}{2}, *)|$ to these sums.

\begin{prp}\label{prime_sum}
Let $C = k^2$ for the family $\mathcal{F} = \{L(s,f)\}_{f \in H_k}$ and $C
= D$ for the family $\{L(s, \chi_{8d})\}_{d \in s(D)}$.  Assume $x = x(C)$ grows
with $C$ in such a way that $\frac{\log x}{\log C}
\to 0$ as $C \to \infty$, but $\log \log x = \log \log C + o(\sqrt{\log
\log C})$.  Define, for $f \in H_k$, \[P(f) = \frac{1}{\sqrt{\log \log k}}
\left( \sum_{n \leq x}
\frac{\Lambda_f(n)}{n^{\frac{1}{2}} \log n} + \frac{1}{2} \log \log
k\right),\] and for $d \in s(D)$,
\[P(d) = \frac{1}{\sqrt{\log \log D}}\left(\sum_{n \leq x}
\frac{\Lambda_{8d}(n)}{n^{\frac{1}{2}}} - \frac{1}{2} \log \log D\right).\] We
have
\begin{equation}\label{conv_of_prime_sum}\frac{1}{|\mathcal{F}|}\sum_{* \in
\mathcal{F}} \delta_{P(*)} \toD
N(0,1), \qquad C \to \infty.\end{equation}
Also, for each $C$ let $\{b_n(C)\}_{n = 1}^\infty$ be a sequence of real
numbers, bounded independently of $C$.  Then 
\begin{equation}\label{mean_error}  \frac{1}{|\mathcal{F}|}\sum_{* \in
\mathcal{F}}\left| \sum_{n <
x^3} \frac{\Lambda_*(n) b_n}{n^{\frac{1}{2}}}\right|^{2} = O(\log^2 x), \qquad
\qquad C \to \infty. \end{equation} 
\end{prp}

\begin{proof}
We show the proof for the family $\mathcal{F} = \{L(s,f)\}_{f \in H_k}$. The
argument for real characters is essentially the same, with the caveat that the
positive mean  when the family is $\{L(s, \chi_{8d})\}_{d
\in s(D)}$ comes from the fact
that
$\left(\frac{8d}{p^2}\right) = 1$ if $p \nmid 8d$.

For (\ref{conv_of_prime_sum}), let $f \in H_k$ and write 
\[P(f) = \frac{1}{\sqrt{\log \log k}}\left[ \sum_{m = 1}^\infty
\frac{1}{m}
\sum_{p < x^{\frac{1}{m}}} \frac{\Lambda_f(p^m)}{p^{\frac{m}{2}} \log p} +
\frac{1}{2} \log \log k \right].\]  Since $\Lambda_f(p^m) = (\lambda_f(p^m) -
\lambda_f(p^{m-2}))\log p$, with $\lambda_f(p^{-1}) = 0$, we have
\begin{align*} P(f) &= \frac{1}{\sqrt{\log \log k}}\sum_{p <
x}\frac{\lambda_f(p)}{p^{ \frac{1}{2}}} + \frac{1}{2\sqrt{\log \log k}} \sum_{p
< \sqrt{x}} \frac{\lambda_f(p^2)-1}{p} + \frac{1}{2} \sqrt{\log \log k} +
o(1)\\ & =  \frac{1}{\sqrt{\log \log k}}\sum_{p <
x}\frac{\lambda_f(p)}{p^{ \frac{1}{2}}} + \frac{1}{2\sqrt{\log \log k}} \sum_{p
< \sqrt{x}} \frac{\lambda_f(p^2)}{p} + o(1), \end{align*} by
Mertens'
theorem for $\sum \frac{1}{p}$.  Regarding the second term, we may assume that
$k$ is sufficiently
large so that $x^2 < k^{2-\delta}$.  Then using orthogonality for $H_k$,
\begin{align*} \frac{1}{|H_k|}\sum_{f \in H_k} \left[\sum_{p <
\sqrt{x}}\frac{\lambda_f(p^2)}{p}\right]^2 &= \frac{1}{|H_k|} \sum_{f
\in H_k} \sum_{p < \sqrt{x}} \frac{1 + \lambda_f(p^2) +
\lambda_f(p^4)}{p^2} + \frac{1}{|H_k|}\sum_{f \in H_k}\sum_{p_1 \neq p_2 <
\sqrt{x}}
\frac{\lambda_f(p_1^2p_2^2)}{p_1p_2}  \\ &=
\sum_{p < \sqrt{x}} \left(\frac{1}{p^2} + \frac{1}{p^3} + \frac{1}{p^4}\right) +
\sum_{p_1 \neq p_2 \leq x} \frac{1}{p_1^2 p_2^2} + O(k^{-\gamma}\log^2 x) =
O(1),
\end{align*} so that, after normalizing by dividing by $\sqrt{\log \log k}$,
this makes a negligible difference to the distribution ((iii) of Lemma  
\ref{conv_dist_lemma}).  It thus suffices to demonstrate that for the prime sum 
\begin{equation}\label{perturbed_convergence} \frac{1}{|H_k|}\sum_{f \in H_k}
\delta_{\tilde{P}(f)} \toD N(0,1); \qquad \qquad \tilde{P}(f) =
\frac{1}{\sqrt{\log \log k}} \sum_{p < x}
\frac{\lambda_f(p)}{p^{\frac{1}{2}}}.\end{equation} This we do by the method of
moments.

 Let $m$ be fixed and assume
now that $k$ is sufficiently large so that $x^{2m} <
k^{2-\delta}$, $x^m < k^{\frac{\gamma}{2}}$. We have
\begin{align*} \frac{1}{|H_k|}\sum_{f \in H_k}\left| \sum_{p <
x}
\frac{\lambda_f(p)}{p^{\frac{1}{2}}} \right|^{2m} &= \sum_{p_1, ..., p_{2m} <x}
\frac{1}{\sqrt{p_1\cdots p_{2m}}} \frac{1}{|H_k|} \sum_{f \in H_k}
\lambda_f(p_1)\cdots \lambda_f(p_{2m}). \end{align*}
When some $p_i$ appears an odd number of times in the list, we see from the
expression (\ref{prime_product}) that $\lambda_f(p_1)\cdots\lambda_f(p_r)$ can
be written as a linear combination of $O_m(1)$ terms $\lambda_f(n_i)$, for
which none of the $n_i$ are squares.  Thus by Lemma \ref{H_k_orthogonality} the
contribution of all such terms is $\ll_m k^{-\frac{\gamma}{2}}$.

Among terms containing each $p_i$ an even number of times, those containing
some $p_i$ at least 4 times contribute $\ll_m (\log \log x)^{m-2}$, which is an
error term.  We are left to consider terms containing each prime exactly twice.
 These contribute 
\[O_m(k^{-\gamma} (\log \log k)^m) +\sum_{\substack{p_1, ..., p_m < x\\
\text{distinct}}}
\frac{1}{p_1 \cdots p_m}\sum_{d|p_1\cdots p_m} \frac{1}{d} = \frac{(2m)!}{2^m
m!} (\log \log x)^m (1 + o_m(1)).\] The claimed convergence in
(\ref{perturbed_convergence}) thus
follows from the fact that the Gaussian distribution is determined by its
moments.

To prove (\ref{mean_error}), assume $x^6 < \min(k^{2-\delta}, k^\gamma)$ and
split the
primes, prime squares, and higher powers as 
\[\left| \sum_{n < x^3} \frac{\Lambda_f(n) b_n}{n^{\frac{1}{2}}}\right|^2
\leq 3 \left[\left|\sum_{p < x^3}\frac{\lambda_f(p)
b_p \log p}{p^{\frac{1}{2}}}\right|^2 + \left|\sum_{p < x^{\frac{3}{2}}}
\frac{O(\log p)}{p}\right|^2 + O(1)  \right]\]
Thus
\begin{align*}\frac{1}{|H_k|}\sum_{f \in H_k} &\left| \sum_{n < x^3}
\frac{\Lambda_f(n)
b_n}{n^{\frac{1}{2}}}\right|^2 \leq \frac{3}{|H_k|}\sum_{f \in
H_k}\left|\sum_{p <
x^3}\frac{\lambda_f(p)
b_p \log p}{p^{\frac{1}{2}}}\right|^2 + O(\log^2 x)
\\ &\leq 3\sum_{p_1, p_2 \leq x^3} \frac{b_{p_1}b_{p_2}\log p_1 \log
p_2}{\sqrt{p_1 p_2}} \frac{1}{|H_k|}\sum_{f \in H_k}
\lambda_f(p_1)\lambda_f(p_2) +
O(\log^2 x)\\
 &\leq \sum_{p \leq x^3} \frac{O(\log^2 p)}{p} + O(k^{-\gamma/2}) + O(\log^2
x) \\&= O(\log^2 x).
\end{align*}

\end{proof}
\section{Proof of main results}
Throughout this section we let $\mathcal{F}$ be a family of $L$-functions,
either $\mathcal{F} =
\{L(s,f)\}_{f \in H_k}$ or $\mathcal{F} = \{L(s,
\chi_{8d})\}_{d \in s(D)}$, and
we let $C = k^2$ or $C = D$ for the conductor in the family.  We also let $*$
stand in for the typical element in the family. 

\begin{proof}[Proof of Theorem \ref{right_of_half}]
Recall that in this theorem, $\sigma = \sigma(C)$ satisfies $\sigma \log C
\to \infty$ while $\sigma \frac{\log C}{\sqrt{\log \log C}} \to 0$ as $C
\to \infty$.  Choose $x=x(C)$
by $\frac{4}{\log x} = \sigma$. Then we have $\log \log x = \log \log C +
O(\log_3
C)$, so that we may appeal to Proposition \ref{prime_sum}. 

Recall that we set \[A(f) = \frac{1}{\sqrt{\log \log k}} \left(
\log \left|L\left(\frac{1}{2} + \frac{4}{\log
 x},f\right)\right|+ \frac{1}{2} \log \log
k\right)\] and \[ A(d) = \frac{1}{\sqrt{\log \log D}} \left( \log
\left|L\left(\frac{1}{2} + \frac{4}{\log x}, \chi_{8d}\right)\right| -
\frac{1}{2} \log \log
D\right).\]
The theorem then asserts
\[
 \frac{1}{|\mathcal{F}|} \sum_{* \in \mathcal{F}} \delta_{A(*)} \toD N(0,1).
\]

By Lemma \ref{zero_prob} there is a set $E \subset \mathcal{F}$ of measure
$o(1)$ such that outside $E$,
$\sigma_{x,*} = \frac{4}{\log x}$.  Thus by Selberg's
approximation (\ref{log_prime_sum}) we have
\begin{equation}\label{expression_for_log} \log L\left(\frac{1}{2} +
\frac{4}{\log x},*\right) = \sum_{n < x^3}
\frac{\Lambda_{x,*}(n)}{n^{\frac{1}{2} + \frac{4}{\log x}} \log n} +
O\left(\frac{1}{\log x} \left|\sum_{n \leq x^3}
\frac{\Lambda_{x,*}(n) }{n^{\frac{1}{2} +
\frac{4}{\log x}}}\right|\right)
+ O\left(\frac{\log C}{\log x}\right)\end{equation}
on a set of measure $1-o(1)$.  Note that the second
error term contributes $o(1)$ to $A(*)$.   Now
\[\sum_{n < x^3}
\frac{\Lambda_{x, *}(n)}{n^{\frac{1}{2} + \frac{4}{\log x}} \log n} = \sum_{n <
x} \frac{\Lambda_*(n)}{n^{\frac{1}{2}}\log n} + \sum_{n <
x} \frac{\Lambda_*(n)}{n^{\frac{1}{2}}\log n}\left(n^{
\frac{-4}{\log x}} - 1\right) + \sum_{x \leq n < x^3}
\frac{\Lambda_{x,*}(n)}{n^{\frac{1}{2} + \frac{4}{\log x}} \log n}.\]
The first term on the right is the prime sum, for which the convergence in
distribution was proved in Proposition \ref{prime_sum}.   
Thus it will suffice to show that the first error term of
(\ref{expression_for_log}) and the second and third terms above do not alter the
distribution.

 Applying (\ref{mean_error}) of Proposition \ref{prime_sum} successively with
corresponding choices of $b_n$, we find that 
\begin{align*}
 \frac{1}{|\mathcal{F}|}\sum_{* \in \mathcal{F}} \left[\frac{1}{\log x}
\sum_{n \leq x^3}
\frac{\Lambda_{x,*}(n) }{n^{\frac{1}{2} +
\frac{4}{\log x}}} \right]^2 = O(1), & \;\; 
 b_n = \frac{a_x(n)}{n^{\frac{4}{\log x}}}\\
\frac{1}{|\mathcal{F}|}\sum_{* \in \mathcal{F}} \left[\sum_{n <
x} \frac{\Lambda_*(n)}{n^{\frac{1}{2}}\log n}\left(n^{
\frac{-4}{\log x}} - 1\right) \right]^2 = O(1), & \;\;
  b_n = \left\{
\begin{array}{ll} \frac{\log
x}{\log n}\left(n^{\frac{-4}{\log x}} - 1\right),  &n \leq  x
\\ 0,& x \leq  n \end{array} \right. \\
\frac{1}{|\mathcal{F}|}\sum_{* \in \mathcal{F}} \left[ \sum_{x \leq n < x^3}
\frac{\Lambda_{x,*}(n)}{n^{\frac{1}{2} + \frac{4}{\log x}} \log n}\right]^2 =
O(1), & \;\;
  b_n = \left\{
\begin{array}{lll}0, && n < x\\ \frac{a_x(n)}{n^{\frac{4}{\log x}}} \frac{\log
x}{\log n}, && x \leq n <
x^3 \end{array} \right..
\end{align*}
Thus by the distribution comparison lemma, Lemma \ref{conv_dist_lemma},
\[\frac{1}{|\mathcal{F
}|} \sum_{* \in \mathcal{F}} \delta_{P(*)} \toD N(0,1) \qquad
\Rightarrow \qquad \frac{1}{|\mathcal{F}|} \sum_{* \in \mathcal{F}}
\delta_{A(*)} \toD N(0,1).\]
\end{proof}

We will deduce Corollary \ref{upper_bound} from Theorem \ref{right_of_half} by
comparing $\log |L(\frac{1}{2}, *)|$ and $\log |L(\frac{1}{2} +
\sigma_{x,*},*)|$. Suppose that $L(\frac{1}{2},*) \neq 0$.  Then choosing a
line integral that makes a small semi-circle to avoid any real zero of $L(s,
*)$,
\begin{align} \notag\log \left|L\left(\frac{1}{2},*\right)\right|-&\log
L\left(\frac{1}{2} +
\sigma_{x,*},*\right) =   \Re \int_{0}^{ \sigma_{x,*}} -
\frac{L'}{L}\left(\frac{1}{2} + \sigma,*\right)d\sigma 
\\&\label{upper_bound_sum} = \log{\frac{\gamma_*(\frac{1}{2} +
\sigma_{x,*})}{\gamma_*(\frac{1}{2})}
} - \sum_{\rho = \frac{1}{2} + \beta + i \gamma:\; \Lambda(\rho,*) = 0}
\Re \int_0^{\sigma_{x,*}} \frac{ \sigma - \beta}{( \sigma
-\beta)^2 + \gamma^2} d\sigma\\
\notag& = O(1) + \frac{1}{2}\sigma_{x,*} \log C  - \frac{1}{2}\sum_{\rho
=
  \frac{1}{2} + \beta + i \gamma: \Lambda(\rho, *)=0}
\log\left(\frac{( \sigma_{x,*} - \beta)^2 + \gamma^2}{ \beta^2 +
\gamma^2}\right),
\end{align}
as follows from the logarithmic derivative of the Hadamard product for
$L(s,*)$ (\ref{hadamard_derivative}), and the approximation
(\ref{log_deriv_gamma})
to the logarithmic derivative of the gamma factor.
When a zero $\rho$ is  far to the right of
the critical line, we pair the contribution of $\rho$ with that of its
reflection $\rho'$ in the line $\Re(s) = \frac{1}{2}$. Combined these contribute
\begin{align} \notag \log &\left[\frac{( \sigma_{x,*} - \beta)^2 +
\gamma^2}{ \beta^2 + \gamma^2} \cdot \frac{(\sigma_{x,*} + \beta)^2 + \gamma^2}{
\beta^2 +
\gamma^2}\right]
\\& \notag \qquad= \log \left[\left( 1 + \frac{(\sigma_{x,*}  -
2\beta)\sigma_{x,*}}{ \beta^2 + \gamma^2}\right) \left(1 +
\frac{(\sigma_{x,*}  + 2 \beta)\sigma_{x,*}}{\beta^2
+\gamma^2}\right)\right]
\\& \label{paired_sum}\qquad = \log\left[1 + \frac{\sigma_{x,*}^2}{ \beta^2 +
\gamma^2}\left( 2 + \frac{\sigma_{x,*}^2 - 4\beta^2}{ \beta^2 +
\gamma^2}\right)\right]
\end{align}
to the sum over zeros in (\ref{upper_bound_sum}).

Corollary \ref{upper_bound} is now deduced by applying the following
Proposition, which is an analog of the upper bound in the
Proposition of \cite{sound_moments},
in the case when RH for the $L$-function is not assumed.

\begin{prp}\label{our_upper_bound}
Continue to let $\mathcal{F} = \{L(s,f)\}_{f \in H_k}$ or
$\mathcal{F} = \{L(s,
\chi_{8d})\}_{d \in s(D)}$, and let $C$ be the conductor of the $L$-functions
in
the family. For $\sigma_{x,*}$ as defined in
(\ref{def_sigma}) we have
\[ \log \left|L\left(\frac{1}{2},*\right)\right| \leq \log L\left(\frac{1}{2} +
\sigma_{x,*},*\right) +
\sigma_{x,*}\log C + O(1).\]

\end{prp}
\begin{proof}
Since $L(\frac{1}{2} + \sigma_{x, *}, *) \neq 0$, if $L(\frac{1}{2}, *) = 0$
then the claim holds, so we may suppose that $L(\frac{1}{2}, *) \neq 0$.  For
zeros $\rho= \frac{1}{2} + \beta + i \gamma$ such that $|\beta| <
\frac{\sigma_{x,*}}{2}$, the contribution of $\beta$ to (\ref{upper_bound_sum})
is clearly negative, so we may assume $\beta > \frac{\sigma_{x,*}}{2}$
and pair $\rho$ and $\rho' = \rho - 2\beta$.  Since $\beta >
\frac{\sigma_{x,*}}{2}$ we have $\rho \not \in
\mathcal{G}_*$ and thus
\[ \gamma > \frac{x^{3\beta}}{\log x} \geq 3\beta \] so that
\begin{equation}\label{bracket_bound}2 \geq 2 + \frac{\sigma_{x,*}^2 - 4
\beta^2}{ \beta^2
+ \gamma^2} \geq 2 -\frac{4\beta^2}{10\beta^2}
\geq \frac{8}{5}\end{equation}  
Hence 
\[ \text{expr. }(\ref{paired_sum}) \geq \log\left[1 + \frac{8}{5}
\frac{\sigma_{x,*}^2}{\beta^2 + \gamma^2}\right] > 0.\]  It follows that the
paired
zeros also contribute a
negative amount to (\ref{upper_bound_sum}), which proves the Proposition.
\end{proof}

\begin{proof}[Deduction of Corollary \ref{upper_bound}]
 Take $x = x(C)$ growing such that $\frac{\log C}{\log x} \to \infty$ but
$\frac{\log C}{\log x \sqrt{\log \log C}} \to 0$ as $C \to \infty$, as in
the proof of Theorem \ref{right_of_half}.  Then as before we have $\sigma_{x,*}
= \frac{4}{\log x}$ except for on a set of measure $o(1)$. It
follows from Proposition \ref{our_upper_bound} that 
\[\frac{\log |L(\frac{1}{2},*)|}{\sqrt{\log \log C}} \leq \frac{\log
L(\frac{1}{2}+
\frac{4}{\log x},*)}{\sqrt{\log
\log C}} + o(1)\]
except on a set of  measure $o(1)$, and the Corollary now follows from the
convergence in distribution of the right hand side proved in Theorem
\ref{right_of_half}.
\end{proof}

We now prove Theorem \ref{conditional} by bounding the negative contribution of
the zeros in (\ref{upper_bound_sum}) by invoking the Low-lying Zero
Hypothesis.

\begin{proof}[Proof of Theorem \ref{conditional}]
 Let $x = x(C)$ and $y = y(C)$ be parameters growing with $C$, satisfying the 
 conditions
\begin{enumerate}
 \item $\frac{\log C}{\log x} \to \infty $
\item  $\frac{\sqrt{\log \log C}}{\log y} \to \infty$, but
$\frac{\log
C}{y \log x} \to 0$
 \item $\frac{\log C \log y}{\log x \sqrt{\log \log C}} \to 0$
\end{enumerate}
as $C \to \infty$.  For instance, these are simultaneously satisfied with \[\log
x = \log C (\log \log C)^{-\frac{1}{4}} ,\qquad y = \log \log C.\]

Since $\frac{\log C}{\log x} \to \infty$,
$\sigma_{x,*} = \frac{4}{\log x}$ except on a set of measure $o(1)$ in
$\mathcal{F}$.  Thus, invoking the Low-lying Zero Hypothesis, there is a subset 
 $\mathcal{F}_0 \subset \mathcal{F}$ of measure $1-o(1)$ in $\mathcal{F}$ which
satisfies, for all $* \in \mathcal{F}_0$, 
$\sigma_{x,*} = \frac{4}{\log x}$ and for all
zeros $\rho = \frac{1}{2} + \beta + i\gamma$ of
$\Lambda(s,*),$  \[|\gamma| > \frac{1}{y \log x}.\]   Restricting to
$\mathcal{F}_0$, by (\ref{upper_bound_sum})
we have
\begin{align*}&\log \left|L\left(\frac{1}{2} + \frac{4}{\log x},*\right)\right|
- \log \left|L\left(\frac{1}{2},*\right)\right| =
\\&\qquad\qquad\qquad O\left(\frac{\log C}{\log x}\right) +
\frac{1}{2}\sum_{\substack{\rho =
  \frac{1}{2} + \beta + i \gamma\\\Lambda(\rho,*) = 0}}
\log\left(\frac{( \frac{4}{\log x} - \beta)^2 + \gamma^2}{ \beta^2 +
\gamma^2}\right).\end{align*}
In view of $\frac{\log C}{\log x} = o(\sqrt{\log \log C})$ this error term does
not alter the distribution, so that, appealing to (iii) of Lemma
\ref{conv_dist_lemma} it suffices to prove
the bound 
\begin{equation*}\tag{\dag}\label{bound} \frac{1}{|\mathcal{F}|} \sum_{* \in
\mathcal{F}_0} \left|\sum_{\rho =
  \frac{1}{2} + \beta + i \gamma:\Lambda(\rho,*) = 0}
\log\left(\frac{( \frac{4}{\log x} - \beta)^2 + \gamma^2}{ \beta^2 +
\gamma^2}\right)\right|^2 = o(\log \log C)\end{equation*} in order to deduce
the theorem from comparison with 
the normal approximation of $\log L(\frac{1}{2} + \frac{4}{\log x}, *)$ proved
in Theorem \ref{right_of_half}.

In the sum over zeros of (\ref{bound}), for $\rho$
with $|\beta| < \frac{2}{\log x}$ we bound the summand in absolute value by 
\begin{align*}
 \left|\log\left(\frac{( \sigma_{x,*} - \beta)^2 + \gamma^2}{ \beta^2 +
\gamma^2}\right)\right| &= \left|\log \left(1- \frac{(\sigma_{x,*} - 
\beta)^2 - \beta^2}{(\sigma_{x,*} - \beta)^2 + \gamma^2}\right)\right|
\\& \leq \left| \log \left(1 - \frac{\sigma_{x,*}^2}{\sigma_{x,*}^2
+ \gamma^2}\right) \right| \\&= \left| \log \left(1 - \frac{1}{1
+ (\frac{\gamma \log x}{4})^2}\right) \right| \ll 
\frac{\log y}{1 + (\frac{\gamma\log x}{4})^2},
\end{align*}
by using $|\gamma| \log x \gg \frac{1}{y}.$
Since $|\beta| \leq \frac{2}{\log x}$, the last quantity is bounded by
\begin{equation}\label{small_beta_bound} \ll \frac{\log y}{\log^2 x}
\frac{1}{(\frac{4}{\log x} - \beta)^2 +\gamma^2}.\end{equation}

For $|\beta| > \frac{2}{\log x} = \frac{\sigma_{x,*}}{2}$ we pair the
contributions of $\rho$ and $\rho' = \rho-2\beta$. By (\ref{paired_sum}) this
combined contribution is
\[\log\left[1 + \frac{\sigma_{x,*}^2}{ \beta^2 +
\gamma^2}\left( 2 + \frac{\sigma_{x,*}^2 - 4\beta^2}{ \beta^2 +
\gamma^2}\right)\right].\]
Since $|\beta| > \frac{\sigma_{x,*}}{2}$ we have $\rho, \rho' \not \in
\mathcal{G}_{x,f}$, and therefore $|\gamma| \geq \frac{x^{3\beta}}{\log x}
\geq 3 \beta$.  By (\ref{bracket_bound}), 
\[2 \geq 2 + \frac{\sigma_{x,*}^2 - 4\beta^2}{ \beta^2 +
\gamma^2} \geq \frac{8}{5},\] and therefore the combined contribution is
bounded in absolute value by
\begin{equation}\label{large_beta_bound}\log \left[1 +
\frac{\sigma_{x,*}^2}{\beta^2 + \gamma^2}\left(2 + \frac{\sigma_{x,*}^2 -4
\beta^2}{\beta^2 + \gamma^2}\right)\right]\ll \frac{\sigma_{x,*}^2}{\beta^2 +
\gamma^2} \ll \frac{1}{\log^2 x} \frac{1}{(\frac{4}{\log x} - \beta)^2 +
\gamma^2}.\end{equation}

Combining (\ref{small_beta_bound}) and (\ref{large_beta_bound}), 
\begin{align*}\left|\sum_{\rho =
  \frac{1}{2} + \beta + i \gamma}
\log\left(\frac{( \frac{4}{\log x} - \beta)^2 + \gamma^2}{ \beta^2 +
\gamma^2}\right) \right| &= O\left( \frac{\log y}{\log^2 x} \sum_\rho
\frac{1}{(\frac{4}{\log x} - \beta)^2 + \gamma^2}\right) \\&= O\left(\frac{\log
y}{\log x}\left| \sum_{n \leq x^3} \frac{\Lambda_{x,*}(n)}{n^{\frac{1}{2}+
\frac{4}{\log x}}}\right|\right) + O\left(\frac{\log y \log C}{\log
x}\right)\end{align*}
by Selberg's bound for sums over zeros in terms of sums
over primes, (\ref{zero_sum_bound}) of Lemma \ref{log_L_off}. 

Since 
\[\frac{\log y \log C}{\log x } = o(\sqrt{\log \log C}),\] the bound
(\ref{bound}) now follows from 
\[\frac{1}{|\mathcal{F}|}{\sum_{* \in \mathcal{F}^*}} \left| \frac{\log
y}{\log x }\sum_{n
\leq x^3} \frac{\Lambda_{x,*}(n)}{n^{\frac{1}{2}+
\frac{4}{\log x}}}\right|^2 = O\left(\log^2 y \right) =
o(\log \log C)\] by (\ref{mean_error}) of Proposition \ref{prime_sum}.
\end{proof}

\bibliographystyle{plain}
\bibliography{central_value}

\end{document}